\documentclass[12pt,a4paper]{article}
\usepackage[a4paper,margin=2cm]{geometry}

\usepackage{amsmath}
\usepackage{amsthm}
\usepackage{subcaption}
\usepackage{amssymb}
\usepackage{enumitem}
\usepackage{graphicx}
\usepackage{nccmath}
\usepackage{setspace}
\usepackage[colorlinks=true, allcolors=blue]{hyperref}
\usepackage{enumerate}
\usepackage{multirow}
\usepackage{float}
\usepackage[dvipsnames]{xcolor}

\newtheorem{theorem}{Theorem}

\newtheorem{lemma}[theorem]{Lemma}

\theoremstyle{definition}

\def \mod#1{{\:({\rm mod}\ #1)}}

\def \Z{\mathbb{Z}}
\def \A{\mathcal{A}}
\def \B{\mathcal{B}}
\def \D{\mathcal{D}}

\def \leq {\leqslant}
\def \geq {\geqslant}

\def \Lc {\overline{L}}

\let\oldproofname=\proofname
\renewcommand{\proofname}{\textup{\textbf{\oldproofname}}}

\title{Completing partial $k$-star designs}
\author{Ajani De Vas Gunasekara\thanks{School of Arts and Sciences, The University of Notre Dame Australia, NSW 2007, Australia \quad ajani.de.vas.gunasekara@nd.edu.au}  \qquad  Daniel Horsley\thanks{School of Mathematics, Monash University, Victoria 3800, Australia \quad danhorsley@gmail.com}}
\date{}

\begin{document}

\maketitle

\begin{abstract}
A \emph{$k$-star} is a complete bipartite graph $K_{1,k}$. A \emph{partial $k$-star design of order $n$} is a pair $(V,\A)$ where $V$ is a set of $n$ vertices and $\A$ is a set of edge-disjoint $k$-stars whose vertex sets are subsets of $V$. If each edge of the complete graph with vertex set $V$ is in some star in $\A$, then $(V,\A)$ is a (complete) \emph{$k$-star design}. We say that $(V,\A)$ is \emph{completable} if there is a $k$-star design $(V,\B)$ such that $\A \subseteq \B$. In this paper we determine, for all $k$ and $n$, the minimum number of stars in an uncompletable partial $k$-star design of order $n$.
\end{abstract}

\section{Introduction}

For many classes of combinatorial design it is the case that all ``small'' partial designs can be completed to a full design, but that this is not the case for larger partial designs. For any such class of designs it is natural to try to determine the size threshold at which uncompletable partial designs first appear. Latin squares give one famous example of this. Smetaniuk \cite{Smetaniuk1981} and Anderson and Hilton \cite{AndersonHilton1983} independently proved a conjecture of Evans \cite{Evans1960} that, for all $n \geq 2$, the smallest number of filled cells in an uncompletable $n \times n$ latin square is exactly $n$. Investigations in this spirit have been undertaken for many objects including block designs \cite{DeVHorEvans}, Hadamard Matrices \cite{Brock1988}, latin hypercubes \cite{BryEtAl}, and matchings in graphs \cite[Chapter 6]{YuLiu}. In this paper we give a solution to the problem for $k$-star designs. It seems this specific problem for $k$-star designs has not previously been considered in the literature, but related problems concerning embeddings of $k$-star designs have been addressed in \cite{DeVHorStar,HoffmanRoberts2014,NobNoc,NobleRichardson2019}. Essentially these papers show that every partial $k$-star design can be completed after some number of new vertices are added to it. For general $k$, the upper bound on this number of new vertices was successively improved in \cite{HoffmanRoberts2014}, \cite{NobleRichardson2019} and \cite{DeVHorStar}. The special case of $k=3$ was dealt with in \cite{NobNoc}.

Let $k \geq 2$ be an integer. A \emph{$k$-star} is a copy of $K_{1,k}$, the complete bipartite graph with parts of size 1 and $k$. A \emph{$k$-star decomposition} of a graph $G$ is a set of $k$-stars whose edge sets partition $E(G)$. For a set $V$, let $K_V$ denote the complete graph with vertex set $V$. A \emph{partial $k$-star design of order $n$} is a pair $(V,\A)$ where $V$ is a set of $n$ vertices and $\A$ is a $k$-star decomposition of some subgraph $G$ of $K_V$. If $G=K_V$, then $(V,\A)$ is a \emph{$k$-star design}. We say a positive integer $n$ is \emph{$k$-admissible} if $\binom{n}{2} \equiv 0 \mod{k}$. Tarsi \cite{Tarsi1979} and Yamamoto et al. \cite{Yamamoto1975} independently proved that, for $n \geq 2$, a $k$-star design of order $n$ exists if and only if $n \geq 2k$ and $n$ is $k$-admissible. A \emph{completion} of a partial $k$-star design $(V,\A)$ is a $k$-star design $(V,\B)$ such that $\A \subseteq \B$. Our main result is as follows.

\begin{theorem} \label{T: main theorem}
Let $k \geq 2$ be an integer. For each $k$-admissible integer $n$ such that $n \geq 2k$, any partial $k$-star design of order $n$ with at most $u(n,k)$ stars has a completion, where
\[u(n,k)=
\begin{cases}
2 \lfloor \frac{n-2}{k} \rfloor-1 & \text{if $n \not \equiv 1\mod{k}$}\\
\frac{2(n-1)}{k} - 2  & \text{if $n \equiv 1\mod{k}$.}
\end{cases}
\]
Furthermore, for each $k$-admissible integer $n > 1$, there is a partial $k$-star design of order $n$ with $u(n,k)+1$ stars that has no completion.
\end{theorem}

In Section~\ref{S:uncompletable} we will exhibit uncompletable partial $k$-star designs to prove the `furthermore' claim in Theorem~\ref{T: main theorem}. Some preliminaries and our overall approach to finding completions of partial $k$-star designs are given in Section~\ref{S:strategy}. Then we will find the required completions for partial $k$-star designs of $k$-admissible orders less than or equal to $3k+1$. For orders larger than $3k+1$, our approach relies on defining a certain function from the vertices of the partial design to the nonenegative integers. We define these functions in Section~\ref{S:suitable} and show that they possess certain desirable properties. In Section~\ref{S:realising} we then use these functions to complete the proof of Theorem~\ref{T: main theorem}. A possible direction for future work is discussed in Section~\ref{S:conclusion}

\section{Uncompletable partial designs}\label{S:uncompletable}

In this section we will exhibit the uncompletable partial $k$-star designs we require for Theorem~\ref{T: main theorem}. We begin by introducing some more definitions and notation that we will use throughout the paper.

In a $k$-star with $k \geq 2$, the vertex of degree $k$ is the \emph{centre} and each other vertex is a \emph{leaf}. Let $G$ be a graph. We denote the complement of $G$ by $\overline{G}$. For a subset $S$ of $V(G)$, we denote the subgraph of $G$ induced by $S$ by $G[S]$. For a given $k$-star decomposition $\D$ of $G$, we can define a function $c: V(G) \rightarrow \Z^{\geq 0}$ called the \emph{central function}, where $c(x)$ is the number of $k$-stars of $\D$ whose centre is $x$ for each $x \in V(G)$. The \emph{leftover} $L$ of a partial $k$-star design $(V,\A)$ is the graph with vertex set $V$ in which each edge of $K_V$ is present if and only if it is not an edge of a star in $\A$. Of course, a partial $k$-star design is completable if and only if there is a $k$-star decomposition of its leftover. Note that $\overline{L}$ is the graph with vertex set $V$ whose edges are exactly the edges in $k$-stars in $\A$. 

As we mentioned above, the orders for which a $k$-star design exists have been completely determined.

\begin{theorem} [\cite{Tarsi1979, Yamamoto1975}]  \label{T:designExistence}
A $k$-star design of order $n$ exists if and only if $n$ is $k$-admissible and $n=1$ or $n \geq 2k$.
\end{theorem}

So, of course, any partial $k$-star design whose order is not $k$-admissible or is in $\{2,\ldots,2k-1\}$ must be uncompletable.
In Lemma~\ref{L:tightness}(i) and (ii) below, we establish the tightness of the bound in Theorem~\ref{T: main theorem}.

\begin{lemma}\label{L:tightness} Let $k \geq 2$ be an integer.
\begin{itemize}
    \item[\textup{(i)}]
For all $k$-admissible integers $n > 1$ such that $n  \not \equiv 1\mod{k}$ there is a partial $k$-star design of order $n$ with $2 \lfloor \frac{n-2}{k} \rfloor$  stars that is not completable.
    \item[\textup{(ii)}]
For all $k$-admissible integers $n > 1$ such that $n  \equiv 1\mod{k}$ there is a partial $k$-star design of order $n$ with $\frac{2n-2}{k}-1$  stars that is not completable.

\end{itemize}
\end{lemma}

\begin{proof}
We first prove (i). Let $(V,\A)$ be a partial $k$-star design of order $n$ with exactly
$2\lfloor \frac{n-2}{k} \rfloor$ stars and central function $c$ such that there exist distinct vertices $x_1, x_2 \in V$ for which $c(x_1) = c(x_2) = \lfloor \frac{n-2}{k} \rfloor$ and $x_1$ and $x_2$ are adjacent in the leftover $L$ of $(V,\A)$. Since $n  \not \equiv 1\mod{k}$, we have $\frac{n-k}{k} \leq \lfloor \frac{n-2}{k} \rfloor \leq \frac{n-2}{k}$ and hence $1 \leq \deg_L(x_i) \leq k-1$ for each $i \in \{1,2\}$. So, in a $k$-star decomposition of $L$, no star can be centred at $x_1$ or $x_2$ and hence no star can contain the edge $x_1x_2$. So $(V,\A)$ cannot be completed.

We now prove (ii). Let $(V,\A)$ be a partial $k$-star design of order $n$ with exactly $\frac{2n-2}{k} - 1$ stars and central function $c$  such that there exist distinct vertices $x_1, x_2,x_3 \in V$ for which $c(x_1) = c(x_2) =  \frac{n-k-1}{k}$, $c(x_3) = 1$, $x_1$ and $x_2$ are leaves of the star centred at $x_3$, and $x_1$ and $x_2$ are adjacent in the leftover $L$ of $(V,\A)$. For each $i \in \{1,2\}$, this implies $\deg_L(x_i) =  k-1$ because  $x_i$ is a leaf of one star and centre of $\frac{n-k-1}{k}$ stars. So, in a $k$-star decomposition of $L$, no star can be centred at $x_1$ or $x_2$ and hence no star can contain the edge $x_1x_2$. So $(V,\A)$ cannot be completed.
\end{proof}

\section{Strategy and preliminary results}\label{S:strategy}

Given Lemma~\ref{L:tightness}, to prove Theorem~\ref{T: main theorem} it remains to show that, for any $k \geq 2$ and $k$-admissible integer $n$, every partial $k$-star design with at most $u(n,k)$ stars has a completion. We begin by observing that we can always add a star to a partial $k$-star design of order $n$ with fewer than $u(n,k)$ stars.

\begin{lemma}\label{L:addStar}
Let $k$ and $n$ be integers such that $k \geq 2$ and $n \geq 2k$, and let $(V,\A')$ be a partial $k$-star design of order $n$ with fewer than $u(n,k)$ stars. The leftover $L'$ of $(V,\A')$ contains a $k$-star.
\end{lemma}

\begin{proof}
Note that $k(u(n,k)-1) < 2n-2k$ and hence $|E(L')|>\binom{n}{2}-ku(n,k)>\frac{1}{2}n(n-5)+2k$. Hence there is a vertex $y$ of $L'$ such that
\[\deg_{L'}(y) > n-5+\mfrac{4k}{n} \geq 2k-3 \geq k-1\]
where the second inequality follows because $n \geq 2k$. So, since $\deg_{L'}(y)$ is an integer, it is at least $k$.
\end{proof}

Given any partial $k$-star design $(V,\A')$ of order $n$ with fewer than $u(n,k)$ stars, we can use Lemma~\ref{L:addStar} to add stars to produce a partial $k$-star design $(V,\A)$ with exactly $u(n,k)$ stars. Any completion of $(V,\A)$ is also a completion of $(V,\A')$. Thus, to prove Theorem~\ref{T: main theorem} it in fact suffices to find completions for just those partial $k$-star designs of $k$-admissible orders $n \geq 2k$ that have exactly $u(n,k)$ stars.

We will now dispense with the case of partial $2$-star decompositions. The graphs admitting a $2$-star decomposition have been completely characterised.

\begin{theorem} [\cite{CaroSch1980}] \label{T: 2-star decomposition}
A $2$-star decomposition of a connected graph $G$ exists if and only if $|E(G)| \equiv 0 \mod{2}$.
\end{theorem}

The following lemma completes the proof of Theorem~\ref{T: main theorem} for $k=2$. Note that $u(n,2)=n-3$ for all integers $n>1$.

\begin{lemma} \label{L: k=2 case}
For all $2$-admissible integers $n > 1$, any partial $2$-star design of order $n$ with
$n-3$ stars is completable.
\end{lemma}

\begin{proof}
Note that $n \geq 4$ since $n$ is $2$-admissible. Let $(V,\A)$ be a partial $2$-star design of order $n$ such that $|\A| =  n-3$ and let $L$ be the leftover of $(V,\A)$. Since $n$ is $2$-admissible, $|E(L)| \equiv 0 \mod{2}$. We will show that $L$ has only one nontrivial component and hence complete the proof by Theorem~\ref{T: 2-star decomposition}. Suppose for a contradiction that $L$ has at least two nontrivial components and let $w$ be the number of vertices in a smallest nontrivial component. Then $2 \leq w \leq \frac{n}{2}$ and $|E(\Lc)| \geq w(n-w) \geq 2(n-2)$. This contradicts the fact that $|E(\Lc)|=2(n-3)$ since $|\A| =  n-3$.
\end{proof}

So, for the remainder of the paper we will concentrate on completing partial $k$-star designs where $k \geq 3$. Throughout we will often write a $k$-admissible integer $n$ as $ak+b$ where $a$ is a nonnegative integer and $b \in \{1,\ldots,k\}$. For each $k \geq 3$ we have that $ak+2$ is not $k$-admissible, and hence that $b \neq 2$. We use this fact tacitly in the remainder of the paper. When $n$ is written in this fashion we have the following equivalent definition of $u(n,k)$, which is often more convenient.
\begin{equation}\label{E:uDefab}
    u(ak+b,k)=
\begin{cases}
2a-2 &\text{if $b=1$}\\
2a-1 &\text{if $b \in \{3,\ldots,k\}$.}
\end{cases}
\end{equation}

We will take advantage of one further reduction of the problem. Call a partial $k$-star design $(V,\A)$ of order $n$ \emph{reducible} if $n \equiv 1 \mod{k}$, $|\A|=u(n,k)$, and there is a vertex which is the centre of at least one star in $\A$ and is not a leaf of any star in $\A$. The following lemma shows that it suffices to find completions only for non-reducible partial $k$-star designs.

\begin{lemma}\label{L:reduction}
Let $k \geq 2$ be an integer. Suppose that, for each $k$-admissible integer $n \geq 2k$, every non-reducible partial $k$-star design of order $n$ with at most $u(n,k)$ stars is completable. Then, for each $k$-admissible integer $n \geq 2k$, every partial $k$-star design of order $n$ with at most $u(n,k)$ stars is completable.
\end{lemma}

\begin{proof}
Let $(V,\A)$ be a reducible partial $k$-star design and say $|V|=ak+1$ where $a \geq 2$ is an integer. Then $\A$ has $u(ak+1,k)=2a-2$ stars. Let $L$ be the leftover of $(V,\A)$ and let $x$ be a vertex which is the centre of at least one star in $\A$ and is not a leaf of any star in $\A$. Then $\deg_L(x)\equiv0 \mod{k}$. Let $\A_x$ be the set of stars in $\A$ that are centred at $x$ and let $\A_x'$ be a set of $\frac{1}{k}\deg_L(x)$ edge-disjoint stars in $L$, all centred on $x$.  Then $(V \setminus \{x\},\A \setminus \A_x)$ is a partial $k$-star design of order $ak$ with at most $u(ak,k)=2a-3$ stars. Now $(V \setminus \{x\},\A \setminus \A_x)$ is non-reducible since its order is $ak$ and so, by our hypotheses, it has a completion $(V \setminus \{x\},\A')$. Then $(V,\A' \cup \A_x \cup \A_x')$ is a completion of $(V,\A)$.
\end{proof}

Throughout the paper, for any function $f$ that assigns real values to vertices in a set $V$ and for any subset $T$ of $V$, we abbreviate $\sum_{x \in T}f(x)$ to $f(T)$. Let $k \geq 2$ be an integer. For any graph $G$ such that $|E(G)| \equiv 0 \mod{k}$ we define a \emph{$k$-precentral function} for $G$ to be a function $p: V(G) \rightarrow \mathbb{Z}^{\geq 0}$ such that $p(V(G))=\frac{1}{k}|E(G)|$. We say a $k$-precentral function for $G$ is \emph{realisable} if it is the central function of some $k$-star decomposition $\D$ of $G$.

Our approach to finding the completions necessary to prove Theorem~\ref{T: main theorem} is based on a key fact which we state as Lemma~\ref{L:realisation} below. This fact is an easy consequence of a result of Tarsi  \cite[Theorem~2]{Tarsi1981} and can also be obtained as a specialisation of results in \cite{Hoffman2004} or \cite{CameronHorsley2020} concerning star decompositions of multigraphs.

\begin{lemma}\label{L:realisation}
Let $p$ be a $k$-precentral function for a graph $G$. Then $p$ is realisable if and only if, for every nonempty proper subset $T$ of $V(G)$, we have $\Delta^+_T \geq \Delta^-_T$ where $\Delta^+_T$ is the number of edges of $G$ that are incident with at least one vertex in $T$ and $\Delta^-_T=kp(T)$.
\end{lemma}

We use the notation $\Delta^+_T$ and $\Delta^-_T$ extensively in the remainder of the paper. We also define $\Delta_T=\Delta^+_T-\Delta^-_T$ so that $\Delta_T$ is nonnegative if and only if $\Delta^+_T \geq \Delta^-_T$. Note that all of this notation is implicitly dependent on a $k$-precentral function $p$ that will always be clear from context. Through the rest of the paper our approach to showing that a partial $k$-star design $(V,\A)$ has a completion will be to define a $k$-precentral function $p$ for its leftover $L$ and then use Lemma~\ref{L:realisation} to show that $p$ is realisable.

\section{Small orders}\label{S:smallOrders}

In this section we complete the proof of Theorem~\ref{T: main theorem} for $n \in \{2k,\ldots,3k+1\}$. Lemmas~\ref{L:2k+1}, \ref{L:3k} and \ref{L:3k+1} below will, respectively, handle the cases where $n \in \{2k,2k+1\}$, where $2k+1<n \leq 3k$, and where $n=3k+1$.

\begin{lemma}\label{L:2k+1}
Let $k \geq 3$ be an integer and let $n$ be a $k$-admissible integer such that $n \in \{2k,2k+1\}$. Any non-reducible partial $k$-star design of order $n$ with $u(n,k)$ stars is completable.
\end{lemma}

\begin{proof}
Let $(V,\A)$ be a partial $k$-star design of order $n$ with $u(n,k)$ stars. If $n=2k+1$, then $|\A|=u(2k+1,k)=2$ and it follows that $(V,\A)$ must be reducible. So, in fact, $n=2k$ and $|\A|=u(2k,k)=1$. By Theorem~\ref{T:designExistence}, there is a $k$-star design of order $n$. By relabelling the vertices of such a design, we can obtain a $k$-star design $(V,\B)$ such that $\A \subseteq \B$.
\end{proof}

\begin{lemma}\label{L:3k}
Let $k \geq 3$ be an integer and let $n$ be a $k$-admissible integer such that $2k+1 < n \leq 3k$. Any partial $k$-star design of order $n$ with three stars is completable.
\end{lemma}

\begin{proof}
Let $b$ be the integer such that $n=2k+b$ and note that $b \in \{3,\ldots,k\}$ since $n$ is $k$-admissible. Let $(V,\A)$ be a partial $k$-star design of order $n$ with $|\A| = 3$, let $L$ be the leftover of $(V,\A)$, and let $c$ be the central function of $\A$. Let $C=\{x \in V: c(x) \geq 1\}$ and note that $|C|\in \{2,3\}$. Let $h=0$ if $|C|=3$ and $h=1$ if $|C|=2$. Let $S$ be a subset of $V \setminus C$ such that $|S|=b+h+\frac{b(b-1)}{2k}-4$ and $\deg_{L}(y) \geq \deg_{L}(x)$ for each $y \in S$ and $x \in V \setminus (C \cup S)$. Note that $S$ is well-defined because $b+h+\frac{b(b-1)}{2k}-4 \in \{0,\ldots,n-3\}$ since $3 \leq b \leq k$ and $n$ is $k$-admissible. We define a $k$-precentral function $p$ by setting $p(x)=2-c(x)$ for each $x \in C$, $p(x)=2$ for each $x \in S$, and $p(x)=1$ for each $x \in V \setminus (C \cup S)$. Note that $p$ is indeed a $k$-precentral function because $p(C)=|C|-h$, $p(V \setminus C)=2k+b-|C|+|S|$ and hence
\[p(V)=2k+b-h+|S|=\mfrac{1}{k}\mbinom{2k+b}{2}-3=\mfrac{|E(L)|}{k}.\]
Let $T$ be a nonempty proper subset of $V$ and let $t=|T|$. By Lemma~\ref{L:realisation} it suffices to show that $\Delta_T \geq 0$. We consider two cases depending on the value of $t$.\smallskip

\noindent \textbf{Case 1.} Suppose $t \geq b$. Then, since $p(T) \leq t+|S|$, we have
\[
\Delta_T^+ \geq \tbinom{t}{2}+t(2k+b-t)-3k, \qquad \Delta_T^- \leq k(t+|S|) = k(t+b+h-4)+\tbinom{b}{2}
\]
and hence $\Delta_T \geq \tfrac{1}{2}(t+1-b)(2k+b-t)-hk$. Noting that $h \leq 1$ and $b \leq t \leq 2k+b-1$, we have $\Delta_T \geq 0$. \smallskip

\noindent \textbf{Case 2.} Suppose $t \leq b-1$. There are $3-c(T)$ stars in $\A$ centred on vertices in $V \setminus T$ and any such star can have at most $t$ leaves in $T$. Also, for each $x \in T$, we have $p(x) \leq 2-c(x)$. So,
\[
\Delta_T^+ \geq \tbinom{t}{2}+t(2k+b-t)-kc(T)-(3-c(T))t, \qquad \Delta_T^- \leq k(2t-c(T))
\]
and hence
\begin{equation}\label{E:3k}
\Delta_T \geq \tfrac{1}{2}t(2b+2c(T)-7-t).
\end{equation}
We can assume that the right side of \eqref{E:3k} is negative for otherwise $\Delta_T \geq 0$ and we are done. In particular, we have $t \geq 2b-6$ since $c(T) \geq 0$. Since $t \leq b-1$, this implies $b \in \{3,4,5\}$. Because $2k+b$ is $k$-admissible we have $b(b-1) \equiv 0 \mod{2k}$ and hence $(b,k) \in \{(3,3),(4,6),(5,5),(5,10)\}$. The following table details, in each of these subcases, the value of $|S|$ and the possible values of $(c(T),t)$ that make the right side of \eqref{E:3k} negative.

\begin{center}
\begin{tabular}{c|c||c|c}
     $b$ & $k$ & $|S|$ & possible $(c(T),t)$  \\ \hline
     3 & 3 & $h$ & $(0,1)$, $(0,2)$, $(1,2)$  \\
     4 & 6 & $1+h$ & $(0,2)$, $(0,3)$  \\
     5 & 5 & $3+h$ & $(0,4)$  \\
     5 & 10 & $2+h$ & $(0,4)$  \\
\end{tabular}
\end{center}

We claim there must be at least $|S|$ vertices in $V \setminus C$ whose degree in $\Lc$ is at most 1. In each subcase this follows from the pigeonhole principle using $|V \setminus C| = 2k+b-3+h$, $\sum_{x \in V \setminus C}\deg_{\Lc}(x) \leq 3k$, the value of $|S|$ given in the table above and $h \leq 1$. Thus, by our definition of $S$ we have $\deg_{\Lc}(x) \leq 1$ for each $x \in S$. Further, $\deg_{\Lc}(x) \leq 3$ for each $x \in V \setminus C$. So, if $c(T)=0$, then
\[
\Delta_T^+ \geq \tbinom{t}{2}+t(2k+b-t)-|S \cap T|-3(t-|S \cap T|), \qquad \Delta_T^- =  k(t+|S \cap T|)
\]
and hence
\begin{equation*}
\Delta_T \geq \tfrac{1}{2}t(2k+2b-t-7)-(k-2)|S \cap T| \geq \tfrac{1}{2}t(2b-t-3) \geq \tfrac{1}{2}t(b-2) > 0
\end{equation*}
where the second inequality follows using $|S \cap T| \leq t$ and the third follows since $t \leq b-1$. So we may assume $c(T) \neq 0$ and hence, from the table, that $(b,k,c(T),t)=(3,3,1,2)$. Then $\deg_{\Lc}(x) \leq 5$ for the unique $x \in C \cap T$, $\deg_{\Lc}(x) \leq 1$ for each $x \in S$, and $\deg_{\Lc}(x) \leq 3$ for each $x \in T \setminus (C \cup S)$. So
\[
\Delta_T^+ \geq 15-5-|S \cap T|-3(1-|S \cap T|), \qquad \Delta_T^- = 3(2+|S \cap T|)
\]
and hence $\Delta_T \geq 1-|S \cap T| \geq 0$ since $|S \cap T| \leq 1$.
\end{proof}

Finally in this section we deal with the case $n=3k+1$. Note that $3k+1$ is $k$-admissible if and only if $k$ is odd.

\begin{lemma}\label{L:3k+1}
Let $k \geq 3$ be an odd integer. Any non-reducible partial $k$-star design of order $3k+1$ with four stars is completable.
\end{lemma}

\begin{proof}
Let $(V,\A)$ be a non-reducible partial $k$-star design of order $3k+1$ with $|\A| = 4$, let $L$ be the leftover of $(V,\A)$, let $c$ be the central function of $\A$, and let $C=\{x \in V: c(x) \geq 1\}$.  Since $(V,\A)$ is non-reducible, $|E(\Lc[C])| \geq |C|$ and hence $|C| \in \{3,4\}$.
Let $h=0$ if $|C|=4$ and $h=1$ if $|C|=3$. Let $S$ be a subset of $V \setminus C$ such that $|S|=h+\frac{1}{2}(3k-7)$ and $\deg_{L}(y) \geq \deg_{L}(x)$ for each $y \in S$ and $x \in V \setminus (S \cup C)$. Note that $S$ is well-defined because $h+\frac{1}{2}(3k-7) \in \{0,\ldots,3k-3\}$ since $k$ is odd. We define a $k$-precentral function $p$ by setting $p(x)=2-c(x)$ for each $x \in C$, $p(x)=2$ for each $x \in S$, and $p(x)=1$ for each $x \in V \setminus (S \cup C)$. Note that $p$ is indeed a $k$-precentral function because $p(C)=|C|-h$, $p(V \setminus C)=3k+1-|C|+|S|$ and hence
\[p(V)=3k+1-h+|S|=\tfrac{3}{2}(3k+1)-4=\tfrac{1}{k}|E(L)|.\]
Let $T$ be a nonempty proper subset of $V$ and let $t=|T|$. By Lemma~\ref{L:realisation} it suffices to show that $\Delta_T \geq 0$. We consider two cases depending on the value of $t$.\smallskip

\noindent \textbf{Case 1.} Suppose $t \geq k+1$. Then, since $p(T) \leq t+|S|$, we have
\[
\Delta_T^+ \geq \tbinom{t}{2}+t(3k+1-t)-4k, \qquad \Delta_T^- \leq k(t+|S|) = \tfrac{1}{2}k(3k+2t+2h-7)
\]
and hence $\Delta_T \geq \tfrac{1}{2}(t-k)(3k+1-t)-hk$. Thus, noting that $h \leq 1$ and $k+1 \leq t \leq 3k$, we have $\Delta_T \geq 0$. \smallskip

\noindent \textbf{Case 2.} Suppose $t \leq k$. There are $4-c(T)$ stars in $\A$ centred on vertices in $V \setminus T$ and any such star can have at most $t$ leaves in $T$. Also, for each $x \in T$, we have $p(x) \leq 2-c(x)$. So
\[
\Delta_T^+ \geq \tbinom{t}{2}+t(3k+1-t)-kc(T)-(4-c(T))t, \qquad \Delta_T^- \leq k(2t-c(T))
\]
and hence
\begin{equation}\label{E:3k+1}
    \Delta_T \geq \tfrac{1}{2}t(2k+2c(T)-7-t).
\end{equation}
We can assume that the right side of \eqref{E:3k+1} is negative for otherwise $\Delta_T \geq 0$ and we are done. In particular, we have $t \geq 2k-6$ since $c(T) \geq 0$. Since $t \leq k$ and $k$ is odd, this implies $k \in \{3,5\}$. The following table details, in both of these subcases, the value of $|S|$ and the possible values of $(c(T),t)$ that make the right side of \eqref{E:3k+1} negative.

\begin{center}
\begin{tabular}{c||c|c}
      $k$ & $|S|$ & possible $(c(T),t)$  \\ \hline
      3 & $1+h$ & $(0,1)$, $(0,2)$, $(0,3)$, $(1,2)$, $(1,3)$  \\
      5 & $4+h$ & $(0,4)$, $(0,5)$  \\
\end{tabular}
\end{center}

We claim there must be at least $|S|$ vertices in $V \setminus C$ whose degree in $\Lc$ is at most 1. In both subcases this follows from the pigeonhole principle using $|V \setminus C| = 3k-3+h$, $\sum_{x \in V \setminus C}\deg_{\Lc}(x) = 4k-|E(\Lc[C])|\leq 4k-4+h$ (recall $|E(\Lc[C])|\geq|C|$), the value of $|S|$ given in the table above, and $h \leq 1$. Thus, by our definition of $S$, we have $\deg_{\Lc}(x) \leq 1$ for each $x \in S$. Further, $\deg_{\Lc}(x) \leq 4$ for each $x \in V \setminus C$. So, if $c(T)=0$, then
\[
\Delta_T^+ \geq \tbinom{t}{2}+t(3k+1-t)-|S \cap T|-4(t-|S \cap T|), \qquad \Delta_T^- = k(t+|S \cap T|)
\]
and hence
\begin{equation*}
\Delta_T \geq \tfrac{1}{2}t(4k-t-7)-(k-3)|S \cap T| \geq \tfrac{1}{2}t(2k-t-1) \geq \tfrac{1}{2}t(k-1) > 0
\end{equation*}
where the second inequality follows using $|S \cap T| \leq t$ and the third follows since $t \leq k$. So we may assume $c(T) \neq 0$ and hence, from the table, that $(k,c(T))=(3,1)$ and $t \in \{2,3\}$. Then $\deg_{\Lc}(x) \leq 6$ for the unique $x \in C \cap T$, $\deg_{\Lc}(x) \leq 1$ for each $x \in S$, and $\deg_{\Lc}(x) \leq 4$ for each $x \in T \setminus (C \cup S)$. So
\[
\Delta_T^+ \geq \tbinom{t}{2}+t(10-t)-6-|S \cap T|-4(t-1-|S \cap T|), \qquad \Delta_T^- = 3(t+|S \cap T|)
\]
and hence $\Delta_T \geq \frac{1}{2}(t-1)(4-t) > 0$ since $t \in \{2,3\}$.
\end{proof}

\section{Suitable precentral functions}\label{S:suitable}

For the leftovers of partial $k$-star designs of orders greater than $3k+1$, we will define the $k$-precentral functions that we aim to realise in a more elaborate way. This section is devoted to explaining how these $k$-precentral functions are defined and proving that they possess certain useful properties. First, however, we require a lemma which tells us that, in the leftover of a partial $k$-star design of order $n$ with $u(n,k)$ stars, not too many vertices can have small degree.

\begin{lemma}\label{L:minimumDegsNew}
Let $k \geq 3$ and $n > 3k+1$ be integers such that $n$ is $k$-admissible, let $(V,\A)$ be a non-reducible partial $k$-star design of order $n$ with at most $u(n,k)$ stars, and let $L$ be the leftover of $(V,\A)$. Then
\begin{itemize}
    \item[\textup{(i)}]
at most one vertex of $L$ has degree at most $k$,
    \item[\textup{(ii)}]
if two adjacent vertices of $L$ have degree less than $2k$, then every other vertex has degree at least $2k$.
\end{itemize}
\end{lemma}

\begin{proof}
Let $u=u(n,k)$. Let $n=ak+b$ where $a \geq 3$ is an integer and $b \in \{1\} \cup \{3,\ldots,k\}$, and let $c$ be the central function of $\A$. We first prove (i). Let $y$ and $z$ be distinct vertices in $V$. We can suppose without loss of generality that $c(y) \leq c(z)$ and that, if $c(y) = c(z)$, then no star centred at $z$ has $y$ as a leaf. It suffices to show that $\deg_L(y) \geq k+1$. Let $\ell=1$ if a star centred at $z$ has $y$ as a leaf and $\ell=0$ otherwise. Now $u-c(y)-c(z)$ stars in $\A$ are centred on vertices in $V \setminus \{y,z\}$ and hence
\begin{equation}\label{E:minDegs}
    \deg_{\Lc}(y) \leq kc(y)+\bigl(u-c(y)-c(z)\bigr) + \ell \leq u+(k-2)c(y)
\end{equation}
where the inequality follows because, from our supposition without loss of generality, either $c(z) \geq c(y)+1$ or $c(y) = c(z)$ and $\ell=0$. If $b=1$, then $u=2a-2$, at least three vertices must have stars centred at them because $(V,\A)$ is not reducible, and hence $c(y) \leq a-2$. If $b \in \{3,\ldots,k\}$, then $u=2a-1$ and hence $c(y) \leq a-1$. Thus, from \eqref{E:minDegs},
\[\deg_{\Lc}(y) \leq
\begin{cases}
k(a-2)+2  & \text{if $b=1$}\\
k(a-1)+1 & \text{if $b \in \{3,\ldots,k\}$.}
\end{cases}\]
Now $\deg_{L}(y) = ak+b-1-\deg_{\Lc}(y)$. So $\deg_{L}(y)  \geq 2k-2 \geq k+1$ if $b=1$ and $\deg_{L}(y)  \geq k+b-2 \geq k+1$ if $b \in \{3,\ldots,k\}$.

Now we prove (ii). Let $S$ be a subset of $V(L)$ such that $|S| = 3$ and $L[S]$ is nonempty. To complete the proof it suffices to show that some vertex in $S$ has degree at least $2k$. For each $x \in S$, let $\ell_x \in \{0,1,2\}$ denote the number of stars that are centred at a vertex in $S$ and have $x$ as a leaf. Let $y$ be a vertex in $S$ such that $c(y) \leq c(x)$ for each $x \in S$ and $\ell_y \leq \ell_z$ for each $z \in S'$ where $S'=\{x \in S: c(x)=c(y)\}$. We will show that $\deg_L(y) \geq 2k$. Let $m=c(S)-3c(y)$ and notice that, by the definition of $y$, $m$ is nonnegative, $|S'|=3$ if $m=0$, and $|S'|=2$ if $m=1$. Hence, because $L[S]$ is nonempty, $(m,\ell_y) \notin \{(0,1),(0,2),(1,2)\}$ by the definition of $y$ and so $m \geq \ell_y$. Thus,
\begin{multline}\label{E:minDegsii}
\deg_L(y) \geq ak+b-1-\ell_y-kc(y)-\bigl(u-c(S)\bigr) = ak+b-1-\ell_y-(k-3)c(y) -u+m \\ \geq ak+b-1-(k-3)\left\lfloor \tfrac{u}{3} \right\rfloor -u.
\end{multline}
where the equality follows because $c(S) = 3c(y)+m$, and the final inequality uses $m \geq \ell_y$ and $c(y) \leq \lfloor\frac{u}{3}\rfloor$. Using $\lfloor \frac{u}{3} \rfloor \leq \frac{u}{3}$ and \eqref{E:uDefab}, we have $\deg_L(y) \geq \frac{1}{3}(a+2)k$ if $b=1$ and $\deg_L(y) \geq \frac{1}{3}(a+1)k+b-1$ if $b\in \{3,\ldots,k\}$. Hence $\deg_L(y) \geq 2k$ when $a \geq 5$. For values of $a$ less than $5$, we consider the lower bound given by \eqref{E:minDegsii} in a number of cases.
\begin{center}
\begin{tabular}{l|l|l}
  case & $u$ & lower bound on $\deg_L(y)$ \\  \hline
  $a=3$, $b\in \{3,\ldots,k\}$ & 5 & $2k+b-3$ \\
  $a=4$, $b=1$ & 6 & $2k$ \\
  $a=4$, $b\in \{3,\ldots,k\}$ & 7 & $2k+b-2$
\end{tabular}
\end{center}
So $\deg_L(y) \geq 2k$ in each case.
\end{proof}

We introduce some notation and definitions relating to $k$-precentral functions. Let $G$ be a graph and let $p$ be a $k$-precentral function for $G$.
We say that $p$ is \emph{proportional} if $p(x) \in \{\lfloor\frac{1}{2k}\deg_G(x)\rfloor,\lceil\frac{1}{2k}\deg_G(x)\rceil\}$ for each $x \in V(G)$. Note that we can always find a proportional $k$-precentral function $p$ for $G$ by first choosing a set $S \subseteq V(G)$ such that $\frac{1}{k}|E(G)|=|S|+\sum_{x \in V(G)}\lfloor\frac{1}{2k}\deg_G(x)\rfloor$ and then setting $p(x)=\lfloor\frac{1}{2k}\deg_G(x)\rfloor$ for $x \in V(G) \setminus S$ and $p(x)=\lfloor\frac{1}{2k}\deg_G(x)\rfloor+1$ for $x \in S$. We define $p^*(x)=p(x)-\frac{1}{2k}\deg_G(x)$ for each $x \in V(G)$. Intuitively, we can think of $p^*(x)$ as the `rounding' that has been applied to $\frac{1}{2k}\deg_G(x)$ to obtain $p(x)$. The definitions of $k$-precentral function and proportional immediately give some basic properties of $p^*$ which we use often and encapsulate in the following lemma.

\begin{lemma}\label{L:gammaStarProperties}
Let $G$ be a graph and let $p$ be a $k$-precentral function for $G$.
\begin{itemize}
    \item[\textup{(i)}]
$p^*(V(G))=0$.
    \item[\textup{(ii)}]
If $p$ is proportional then, for each $x \in V(G)$, $p^*(x)\in \{\frac{-2k+1}{2k},\frac{-2k+2}{2k},\ldots,\frac{2k-1}{2k}\}$.
\end{itemize}
\end{lemma}

We further define a proportional $k$-central function $p$ for $G$ as \emph{minimal} if, among all the proportional $k$-central functions for $G$, $p$ has a minimum value of $\sum_{x \in V(G)}|p^*(x)|$. The following simple lemma will be useful. 

\begin{lemma}\label{L:diffAtMost1}
If $G$ is a graph and $m$ is a minimal $k$-precentral function on $G$, then $m^*(y)-m^*(z) \leq 1$ for all $y,z \in V(G)$.
\end{lemma}

\begin{proof}
Suppose otherwise that there are vertices $y$ and $z$ of $G$ such that $m^*(y)-m^*(z) > 1$. By Lemma~\ref{L:gammaStarProperties}(ii) it must be the case that $m^*(y)$ is positive and hence $m(y) \geq 1$, and that $m^*(z)$ is negative. Define $k$-precentral function $p$ by $p(y)=m(y)-1$, $p(z)=m(z)+1$ and $p(x)=m(x)$ for each $x \in V(G) \setminus \{y,z\}$. Then $p^*(y)=m^*(y)-1$, $p^*(z)=m^*(z)+1$ and $p^*(x)=m^*(x)$ for each $x \in V(G) \setminus \{y,z\}$. So $p$ is proportional because $-1 < p^*(x) < 1$ for each $x \in V(G)$. Furthermore, \[|p^*(y)|+|p^*(z)|=\bigl(1-m^*(y)\bigr)+m^*(z)+1=2-\bigl(m^*(y)-m^*(z)\bigr) < 1\]
whereas $|m^*(y)|+|m^*(z)|=m^*(y)-m^*(z)>1$. Thus $\sum_{x \in V(G)}|p^*(x)|<\sum_{x \in V(G)}|m^*(x)|$, contradicting the fact that $m$ is minimal.
\end{proof}

Sometimes a minimal $k$-precentral function can have an obvious `flaw' that prevents it from being realisable. Let $m$ be a minimal $k$-precentral function for a graph $G$ with vertex set $V$. We say a vertex $y \in V$ is \emph{bad under $m$} if $\deg_G(y) < k$ and $m(y) = 1$, and we say an edge $y_1y_2 \in E(G)$ is \emph{bad under $m$} if $m(y_1)=m(y_2)=0$. When one of these flaws arises, we can repair it by modifying the function slightly. We say a $k$-precentral function $s$ for $G$ is \emph{suitable} if one of the following holds.
\begin{itemize}
    \item[(i)]
$s=m$ for a minimal $k$-precentral function $m$ under which no vertex or edge is bad.
    \item[(ii)]
$s$ is obtained from a minimal $k$-precentral function $m$ for $G$ under which $y \in V$ is bad by setting $s(y)=m(y)-1=0$, $s(z)=m(z)+1$ for some vertex $z$ such that $m^*(z)=\min\{m^*(x): x \in V\}$, and $s(x)=m(x)$ for each $x \in V \setminus \{y,z\}$. In this case we say $m$ is \emph{obtained from $s$ by repairing vertex $y$}. Note $z \neq y$ by Lemma~\ref{L:gammaStarProperties}(i) since $m^*(y)$ is positive.
    \item[(iii)]
$s$ is obtained from a minimal $k$-precentral function $m$ for $G$ under which an edge $y_1y_2 \in E(G)$ with $\deg_G(y_1) \leq \deg_G(y_2)$ is bad by setting $s(y_2)=m(y_2)+1=1$, $s(z)=m(z)-1$ for some vertex $z \in V$ such that $m^*(z)=\max\{m^*(x): x \in V\}$, and $s(x)=m(x)$ for each $x \in V \setminus \{y_2,z\}$. In this case we say $m$ is \emph{obtained from $s$ by repairing edge $y_1y_2$ at $y_2$}. Note that $\deg_G(y_2) \geq 1$, so $m^*(y_2)$ is negative and hence $m^*(z)$ is positive by Lemma~\ref{L:gammaStarProperties}(i) and $m(z) \geq 1$.
\end{itemize}

From this definition it is clear that, for a given graph $G$ with $|E(G)| \equiv 0 \mod{k}$, we can always obtain a suitable $k$-precentral function for $G$ by taking a minimal $k$-precentral function for $G$ and, if necessary, modifying it appropriately. However, it seems possible that a suitable $k$-precentral function could still be bad at some vertex or edge. The next lemma shows that this cannot happen and also that suitable $k$-precentral functions are always proportional.

\begin{lemma}\label{L:suitableIsGood}
Let $k \geq 3$ and $n > 3k+1$ be integers, and let $L$ be the leftover of a partial $k$-star design $(V,\A)$ of order $n$ with $u(n,k)$ stars. If $s$ is a suitable $k$-precentral function on $L$, then $s$ is proportional and no vertex or edge of $L$ is bad under $s$.
\end{lemma}

\begin{proof}
If $s=m$ for a minimal $k$-precentral function $m$ for $L$ under which no vertex or edge is bad, then the result holds trivially, so we may suppose otherwise. We will make repeated use of the fact that a proportional $k$-precentral function does not assign 0 to any vertex of degree at least $2k$.

Suppose that $s$ is obtained from a minimal $k$-precentral function $m$ for $L$ by repairing vertex $y$. Then $s(y)=m(y)-1=0$, $s(z)=m(z)+1$ for some vertex $z$ such that $m^*(z)=\min\{m^*(x): x \in V\}$, and $s(x)=m(x)$ for each $x \in V \setminus \{y,z\}$. It can be seen that $s$ is proportional because $m$ is proportional, $m^*(y) > 0$ and, by Lemma~\ref{L:gammaStarProperties}(i), $m^*(z) < 0$. By Lemma~\ref{L:minimumDegsNew}(i), for each $x \in V \setminus \{y\}$ we have that $\deg_L(x) \geq k+1$ and hence $s$ is not bad on any vertex. We claim that $s(x) \geq 1$ for each $x \in V \setminus \{y\}$ and hence that $s$ is not bad on any edge. If $\deg_L(x) \geq 2k$, this follows because $s$ is proportional and, if $k+1 \leq \deg_L(x) < 2k$, then $s(x) \geq m(x)=1$ since otherwise we would have $m^*(x) < -\frac{1}{2}$ and $m^*(y)>\frac{1}{2}$ contradicting Lemma~\ref{L:diffAtMost1}.

Suppose that $s$ is obtained from a minimal $k$-precentral function $m$ for $L$ by repairing edge $y_1y_2$ at $y_2$. Then $s(y_2)=m(y_2)+1=1$, $s(z)=m(z)-1$ for some vertex $z \in V$ such that $m^*(z)=\max\{m^*(x): x \in V\}$, and $s(x)=m(x)$ for each $x \in V \setminus \{y_2,z\}$. It can be seen that $s$ is proportional because $m$ is proportional, $m^*(y_2) < 0$ and, by Lemma~\ref{L:gammaStarProperties}(i), $m^*(z) > 0$. Now $\deg_L(y_1) \leq \deg_L(y_2) < 2k$ and hence, by Lemma~\ref{L:minimumDegsNew}, $\deg_L(y_2) \geq k+1$ and $\deg_L(x) \geq 2k$ for each $x \in V(L) \setminus \{y_1,y_2\}$. Noting $s(y_1)=m(y_1)=0$, it follows that $s$ is not bad at any vertex of $L$. Further, since $s$ is proportional, $s(x) \geq 1$ for each $x \in V \setminus \{y_1\}$ and so $s$ is not bad on any edge of $L$.
\end{proof}

Finally in this section, we establish upper bounds on the sum of the values of a suitable $k$-precentral function over a subset of the vertices in a graph. This will be vital to showing that our suitable $k$-precentral functions are in fact realisable.

\begin{lemma}\label{L:gammaStarSum}
Let $p$ be a $k$-precentral function for a $n$-vertex graph $G$ on vertex set $V$, let $T$ be a nonempty proper subset of $V$, and let $t=|T|$.
\begin{itemize}
    \item[\textup{(i)}]
If $p$ is minimal, then $p^*(T) \leq \frac{t(n-t)}{n}$.
    \item[\textup{(ii)}]
If $p$ is suitable, then
\[p^*(T) \leq
\begin{cases}
\frac{t(2n-2t-1)}{2(n-1)}\quad   &\mbox{if $t < \frac{n}{2}$}\\
\frac{(2t-1)(n-t)}{2(n-1)} &\mbox{if $t \geq \frac{n}{2}$.}
\end{cases}
\]
\end{itemize}
\end{lemma}

\begin{proof}
Let $V=V(G)$ and $t=|T|$. We first prove (i). Suppose for a contradiction that $p$ is minimal and $p^*(T) > \frac{t(n-t)}{n}$. By Lemma~\ref{L:gammaStarProperties}(i), we also have $p^*(V \setminus T) < -\frac{t(n-t)}{n}$. But then there must exist $y \in T$ with $p^*(y)>\frac{n-t}{n}$ and $z \in V\setminus T$ with $p^*(z)<-\frac{t}{n}$, contradicting Lemma~\ref{L:diffAtMost1}.

We now prove (ii). We do not retain any notation from our proof of (i). Suppose that $p$ is suitable. If $p$ is minimal, then the result follows by (i), noting that $\frac{t(n-t)}{n} < \frac{t(2n-2t-1)}{2(n-1)}$ when $0 < t < \frac{n}{2}$ and $\frac{t(n-t)}{n}\leq \frac{(2t-1)(n-t)}{2(n-1)}$ when $\frac{n}{2} \leq t < n$. (To see this, note that $\frac{t(2n-2t-1)}{2(n-1)}-\frac{t(n-t)}{n}=\frac{t(n-2t)}{2n(n-1)}$ and $\frac{(2t-1)(n-t)}{2(n-1)}-\frac{t(n-t)}{n}=\frac{(n-t)(2t-n)}{2n(n-1)}$.) So we may assume that $p$ is not minimal. Thus, by the definition of suitable, we must be in one of the following two cases.\smallskip

\noindent \textbf{Case 1.} Suppose that $p$ is obtained from a minimal $k$-precentral function $m$ for $G$ by repairing vertex $y$. Then $\deg_G(y) \leq k-1$, $p(y)=m(y)-1=0$, $p(z)=m(z)+1$ for some vertex $z$ such that $m^*(z)=\min\{m^*(x): x \in V\}$, and $m(x)=p(x)$ for each $x \in V(G) \setminus \{y,z\}$. Let $\ell=m^*(z)$ and note that $-1< \ell <0$ by Lemma~\ref{L:gammaStarProperties}(i) since $m^*(y)$ is positive and that $p^*(z)=\ell+1$.

Since $m$ is minimal, by Lemma~\ref{L:diffAtMost1} we have $\ell \leq m^*(x) \leq \ell+1$ for each $x \in V$. Thus $p^*(x) \leq \ell+1$ for each $x \in V$ and, in particular,
\begin{equation}\label{E:tBound1}
p^*(T) \leq t(\ell+1).
\end{equation}
Further, $p^*(y)= m^*(y)-1 \leq \ell$ and $p^*(x) \geq \ell$ for each $x \in V \setminus \{y\}$. Thus, $p^*(V\setminus T) \geq (n-t-1)\ell+p^*(y)$ and hence, using Lemma~\ref{L:gammaStarProperties}(i),
\begin{equation}\label{E:tBound2}
p^*(T) \leq -(n-t-1)\ell-p^*(y).
\end{equation}
Now,
\[p^*(T) \leq \mfrac{t(n-t-1-p^*(y))}{(n-1)} < \mfrac{t(2n-2t-1)}{2(n-1)}\]
where the first inequality follows from \eqref{E:tBound1} if $\ell \leq -\frac{1}{n-1}(t+p^*(y))$ and from \eqref{E:tBound2} otherwise, and the second follows using $p^*(y)=-\frac{1}{2k}\deg_G(y)>-\frac{1}{2}$. This completes the proof because $\frac{t(2n-2t-1)}{2(n-1)} \leq \frac{(2t-1)(n-t)}{2(n-1)}$  when $t \geq \frac{n}{2}$. \smallskip

\noindent \textbf{Case 2.} Suppose that $p$ is obtained from a minimal $k$-precentral function $m$ for $G$ by repairing edge $y_1y_2$ at $y_2$. Then $p(y_2)=m(y_2)+1=1$, $p(z)=m(z)-1$ for some vertex $z \in V$ such that $m^*(z)=\max\{m^*(x): x \in V\}$, and $p(x)=m(x)$ for each $x \in V \setminus \{y_2,z\}$. Note $\deg_L(y_1) \leq \deg_L(y_2)$ and hence, by Lemma~\ref{L:minimumDegsNew}(i), $\deg_L(y_2) \geq k+1$. Let $g=m^*(z)$ and note that $0< g < 1$ by Lemma~\ref{L:gammaStarProperties} since $m^*(y_2)$ is negative and that $p^*(z)=g-1$.

Since $m$ is minimal, by Lemma~\ref{L:diffAtMost1} we have $g-1 \leq m^*(x) \leq g$ for each $x \in V$. Thus $p^*(x) \geq g-1$ for each $x \in V$, so $p^*(V\setminus T) \geq (n-t)(g-1)$ and, using Lemma~\ref{L:gammaStarProperties}(i),
\begin{equation}\label{E:tBound1Case2}
p^*(T) \leq (n-t)(1-g).
\end{equation}
Further, $p^*(y_2)= m^*(y_2)+1 \geq g$ and $p^*(x) \leq g$ for each $x \in V \setminus \{y_2\}$. Thus,
\begin{equation}\label{E:tBound2Case2}
p^*(T) \leq (t-1)g+p^*(y_2).
\end{equation}
Now,
\[p^*(T) \leq \mfrac{(n-t)(t-1+p^*(y_2))}{(n-1)} < \mfrac{(2t-1)(n-t)}{2(n-1)}\]
where the first inequality follows from \eqref{E:tBound1Case2} if $g \geq \frac{1}{n-1}(n-t-p^*(y_2))$ and from \eqref{E:tBound2Case2} otherwise, and the second follows using $p^*(y_2)=1-\frac{1}{2k}\deg_G(y_2) < \frac{1}{2}$. This completes the proof because $\frac{(2t-1)(n-t)}{2(n-1)} < \frac{t(2n-2t-1)}{2(n-1)}$  when $t < \frac{n}{2}$.
\end{proof}

\section{Realising suitable precentral functions}\label{S:realising}

In this section we complete the proof of Theorem~\ref{T: main theorem} by showing that a suitable $k$-precentral function on the leftover $L$ of one of the relevant partial designs is realisable. In view of Lemma~\ref{L:realisation}, we do this by taking a nonempty proper subset $T$ of $V(L)$ and showing that $\Delta_T \geq 0$ by dividing into cases according to the size of $T$. Our next lemma covers the case when $T$ is neither very large nor very small. We then deal with very large $T$ in Lemma~\ref{L:largeT} and very small $T$ in Lemma~\ref{L:smallT}.

\begin{lemma}\label{L:midrangeT}
Let $k \geq 3$ and $n > 3k+1$ be integers such that $n$ is $k$-admissible. Let $(V,\A)$ be a non-reducible partial $k$-star design of order $n$ with $u(n,k)$ stars, let $L$ be the leftover of $(V,\A)$, and let $s$ be a suitable $k$-precentral function for $L$. If $T$ is a subset of $V(L)$ with $5 \leq |T| \leq n-5$, then $\Delta_T \geq 0$.
\end{lemma}

\begin{proof}
Let $u=u(n,k)$, 
let $t=|T|$, let $e$ be the number of edges incident with two vertices in $T$, and let $c$ be the number of edges incident with exactly one vertex in $T$. Then we have $\Delta^+_T = e+c$ and $\sum_{x \in T} \deg_L(x) = 2e+c$. Using these, together with the fact that $s(x) = \frac{1}{2k}\deg_L(x)+s^*(x)$ for each $x \in V$, we obtain
  \begin{equation} \label{E: deltaEq}
     \Delta_T = e+c-ks(T) = e+c-\left(\mfrac{1}{2}\sum_{x \in T}\deg_L(x)+ ks^*(T)\right) = \mfrac{c}{2} - ks^*(T).
 \end{equation}
Also,
\begin{equation} \label{E: cbound}
    c \geq t(n-t)-ku \geq t(n-t)-(2n-k-5)
\end{equation}
where for the final inequality we used the definition of $u$ and, in the case $n \equiv 1 \mod{k}$, $k \geq 3$. We consider two cases according to the value of $t$.\smallskip

\noindent\textbf{Case 1.} Suppose that $t \geq \frac{n}{2}$. Then $s^*(T) \leq \frac{(2t-1)(n-t)}{2(n-1)}$ by Lemma~\ref{L:gammaStarSum}(ii). Using this and \eqref{E: cbound} in \eqref{E: deltaEq}, we obtain
\begin{equation}\label{E:DeltaGenBigt}
        \Delta_T  \geq \mfrac{c}{2}-\mfrac{k(2t-1)(n-t)}{2(n-1)} \geq \mfrac{t(n-t)+5-2n}{2}-\mfrac{k}{2}\left(\mfrac{(2t-1)(n-t)}{n-1}-1\right).
\end{equation}
The coefficient of $k$ in \eqref{E:DeltaGenBigt} is nonpositive since $(2t-1)(n-t) \geq n-1$ using $\frac{n}{2} \leq t \leq n-5$ and $n \geq 11$.
So we can substitute $k < \frac{n-1}{3}$ into \eqref{E:DeltaGenBigt} to obtain
\[\Delta_T > \mfrac{(t-4)(n-t-5)}{6}-1 \geq -1.\]
Thus $\Delta_T$ is nonnegative, since it is an integer.
\smallskip

\noindent\textbf{Case 2.} Suppose that $t <  \frac{n}{2}$. Then $s^*(T) \leq \frac{t(2n-2t-1)}{2(n-1)}$ by Lemma~\ref{L:gammaStarSum}(ii). Using this and \eqref{E: cbound} in \eqref{E: deltaEq}, we obtain
\begin{equation}\label{E:DeltaGenSmallt}
    \Delta_T  \geq \mfrac{c}{2}-\mfrac{tk(2n-2t-1)}{2(n-1)} \geq \mfrac{t(n-t)+5-2n}{2}-\mfrac{k}{2}\left(\mfrac{t(2n-2t-1)}{n-1}-1\right).
\end{equation}
The coefficient of $k$ in \eqref{E:DeltaGenSmallt} is nonpositive since $t(2n-2t-1) \geq n-1$ using $5 \leq t < \frac{n}{2}$ and $n \geq 11$.
So we can substitute $k < \frac{n-1}{3}$ into \eqref{E:DeltaGenSmallt} and simplify to obtain
\[\Delta_T > \mfrac{(t-5)(n-t-4)}{6}-1 \geq -1.\]
Thus $\Delta_T$ is nonnegative, since it is an integer.
\end{proof}

\begin{lemma} \label{L:largeT}
Let $k \geq 3$ and $n > 3k+1$ be integers. Let $(V,\A)$ be a partial $k$-star design of order $n$ with $u(n,k)$ stars, let $L$ be its leftover, and let $s$ be a suitable $k$-precentral function for $L$. If $T$ is a proper subset of $V$ with $|T| \geq n-4$, then $\Delta_T \geq 0$.
\end{lemma}

\begin{proof}
Let $e=|E(L[V \setminus T])|$ and note $e \leq 6$ since $|V \setminus T| \leq 4$. We have $\Delta^+_T=|E(L)|-e$. Since $ks(V)=|E(L)|$, we have $\Delta^-_T=|E(L)|-ks(V\setminus T)$. Thus $\Delta_T=ks(V\setminus T)-e$. This is obviously nonnegative if $e=0$. If $e \in \{1,2,3\}$, then $s(V\setminus T) \geq 1$ because no edge in $L[V \setminus T]$ is bad under $s$, and hence $\Delta_T$ is nonnegative. If $e \in \{4,5,6\}$, then $L[V \setminus T]$ must contain two nonadjacent edges and hence $s(V\setminus T) \geq 2$ because no edge in $L[V \setminus T]$ is bad under $s$. So again $\Delta_T$ is nonnegative.
\end{proof}

\begin{lemma}\label{L:smallT}
Let $k \geq 3$ and $n > 3k+1$ be integers. Let $(V,\A)$ be a non-reducible partial $k$-star design of order $n$ with $u(n,k)$ stars, let $L$ be its leftover, and let $s$ be a suitable $k$-precentral function for $L$. If $T$ is a nonempty subset of $V$ with $|T| \leq 4$, then $\Delta_T \geq 0$.
\end{lemma}

\begin{proof}
Let $u=u(n,k)$ and let $c$ be the central function of $\A$. Say $n=ak+b$ where $a \geq 3$ and $b \in \{1\} \cup \{3,\ldots,k\}$. Note that $(a,b) \neq (3,1)$ since $n > 3k+1$. Let $t=|T|$ and $e=|E(L[T])|$. Note that a star centred at a vertex in $V \setminus T$ contributes at most $k$ toward $\sum_{x \in T}\deg_{\Lc}(x)$ and a star centred at a vertex in $T$ contributes $k+i$, where $i$ is the number of leaves it has in $T$. Thus $\sum_{x \in T}\deg_{\Lc}(x) \leq ku+\binom{t}{2}-e$ and we have $\sum_{x\in T}\deg_L(x) \geq t(n-1)-ku-\binom{t}{2}+e$. So $\Delta^+_T \geq t(n-1)-ku-\binom{t}{2}$ and hence $\Delta_T \geq t(n-1)-ku-\tbinom{t}{2}-ks(T)$. We may assume that the right hand side of this inequality is negative for otherwise we are finished and hence, using $n=ak+b$ and \eqref{E:uDefab} we have
\begin{align}\label{E:SmallDeltaTbound}
    \mbinom{t}{2} >
\begin{cases}
    k(a(t-2)+2-s(T))& \text{if $b=1$,}\\
    k(a(t-2)+1-s(T))+t(b-1)& \text{if $b \geq 3$.}
\end{cases}
\end{align}

For each $x \in V$, let $\delta(x)= \deg_{L}(x)-ks(x)$. Since $s$ is proportional and no vertex is bad under $s$ we have
\begin{equation}\label{E:littleDeltaBounds}
\delta(x) \geq
\begin{cases}
    \deg_{L}(x) & \text{if $\deg_{L}(x) \in \{0,\ldots,k-1\}$}\\
    \deg_{L}(x)-k & \text{if $\deg_{L}(x)\in \{k,\ldots,2k\}$}\\
    \deg_{L}(x)-2k & \text{if $\deg_{L}(x)\in \{2k+1,\ldots,4k\}$}\\
    k+1 & \text{if $\deg_{L}(x) \geq 4k+1$.}
\end{cases}
\end{equation}
In particular, note that $\delta(x) \geq 0$, that $\delta(x)=0$ if and only if $(\deg_{L}(x),s(x)) \in \{(0,0),(k,1)\}$, and that $\delta(x)=1$ if and only if $(\deg_{L}(x),s(x)) \in \{(1,0),(k+1,1),(2k+1,2)\}$. Also, since $\Delta^+_T = (\sum_{x \in T}\deg_{L}(x))-e$, we have
$\Delta_T = \delta(T)-e$. So to complete the proof it suffices to assume $\delta(T) \leq e-1$ and obtain a contradiction. We consider cases according to the value of $t$.\smallskip

\noindent\textbf{Case 1.} Suppose that $t\in \{1,2\}$. If $t=1$, then we immediately have the contradiction $\delta(T) \leq e-1 = -1$. If $t=2$, then it must be that $e=1$ and $\delta(x)=0$ for each $x \in T$, but this is impossible by Lemmas~\ref{L:minimumDegsNew}(i) and \eqref{E:littleDeltaBounds}.\smallskip

\noindent\textbf{Case 2.} Suppose that $t=3$. Then $e\in\{1,2,3\}$. In fact, by Lemmas~\ref{L:minimumDegsNew}(i) and \eqref{E:littleDeltaBounds}, it must be that $e=3$, $(\deg_{L}(z),s(z)) \in \{(0,0),(k,1)\}$ for some $z \in T$, and $(\deg_{L}(x),s(x)) \in \{(k+1,1),(2k+1,2)\}$ for each $x \in T \setminus\{z\}$. So $s(T) \leq 5$. Substituting this and $t=3$ into  \eqref{E:SmallDeltaTbound}, we have that $k(a-3) < 3$ if $b=1$ and $k(a-4)+3b < 6$ if $b \geq 3$. The former immediately contradicts $a \geq 4$ and hence, from the latter, we must have $a=3$ and $3 \leq b \leq \frac{1}{3}(k+5)$.
So $u=5$ and, since $e=3$ and $\deg_{\Lc}(z) \in \{2k+b-1,3k+b-1\}$, it must be the case that $c(z) \geq 2$ and $z$ is a leaf of at least two stars centred at vertices in $V \setminus T$. So, for some $y \in T \setminus \{z\}$, we have $c(y)=0$ and that $y$ is a leaf of at most three stars (note that $y$ is not a leaf of any star centred at $z$ because $e=3$). This contradicts $\deg_{\Lc}(y) \in \{k+b-2,2k+b-2\}$.  \smallskip

\noindent\textbf{Case 3.} Suppose that $t=4$. Then $0 \leq e \leq 6$ By Lemma~\ref{L:minimumDegsNew}(i) we have $\delta(y) = 0$ for at most one  vertex $y \in T$. Hence it cannot be that $\delta(x) \geq 4$ for any $x \in T$ since $\delta(T) \leq e-1 \leq 5$. So we may assume, for each $x \in T$, that $\delta(x) \leq 3$ and hence that $\deg_L(x) \leq 3k$ by \eqref{E:littleDeltaBounds}. Since $s$ is proportional, we have $s(x) \leq 2$ for each $x \in T$. Hence $s(T) \leq 8$. Substituting $t=4$ into  \eqref{E:SmallDeltaTbound}, we have that $k(2a+2-s(T)) < 6$ if $b=1$ and $k(2a+1-s(T))+4b < 10$ if $b \geq 3$. In the former case we have an immediate contradiction to $a \geq 4$ and $s(T) \leq 8$, and hence, from the latter, we must have $a=3$, $s(T)=8$ and $3 \leq b \leq \frac{1}{4}(k+9)$. So $s(x)=2$ for each $x \in T$. Thus, for each $x \in T$,  we have $\deg_L(x)=2k+\delta(x)$ and hence $s^*(x)=1-\frac{1}{2k}\delta(x)$ or, equivalently, $\delta(x) = 2k-2ks^*(x)$. So, by Lemma~\ref{L:gammaStarSum},
\[\delta(T)=8k-2k\sum_{x \in T}s^*(x) \geq 8k-2k\cdot\mfrac{4(2n-9)}{2(n-1)} = \mfrac{28k}{n-1}>7>e\]
where the second last inequality follows because $n<4k+1$ since $a=3$.
\end{proof}

Theorem~\ref{T: main theorem} now follows by combining the results we have obtained.

\begin{proof}[\textbf{\textup{Proof of Theorem~\ref{T: main theorem}}}]
Fix $n$ and $k$ and let $u=u(n,k)$. By Lemma~\ref{L:tightness}, there is an uncompletable partial $k$-star design of order $n$ with $u+1$ stars. So it suffices to show that any partial $k$-star design of order $n$ with at most $u$ stars has a completion. In fact, by Lemma~\ref{L:reduction} it suffices to show that any non-reducible partial $k$-star design of order $n$ with at most $u$ stars has a completion.

Let $(V,\A')$ be a non-reducible partial $k$-star design of order $n$ with at most $u$ stars and let $L'$ be the leftover of $(V,\A')$. If necessary, use Lemma~\ref{L:addStar} to add stars to $(V,\A')$ to form a partial $k$-star design $(V,\A)$ of order $n$ with exactly $u$ stars. It suffices to show that $(V,\A)$ has a completion or, equivalently, to show there is $k$-star decomposition of its leftover $L$.

If $k=2$, then $(V,\A)$ has a completion by Lemma~\ref{L: k=2 case}. If $k\geq 3$ and $n \leq 3k+1$, then $(V,\A)$ is completable by Lemma~\ref{L:2k+1}, Lemma~\ref{L:3k} or Lemma~\ref{L:3k+1}. So we may assume $k \geq 3$ and $n > 3k+1$. Let $L$ be the leftover of $(V,\A)$, let $s$ be a suitable $k$-precentral function for $L$, and let $T$ be a nonempty proper subset of $V$. By Lemma~\ref{L:realisation}, it suffices to show that $\Delta_T \geq 0$. This follows from Lemmas~\ref{L:midrangeT}, \ref{L:largeT} and \ref{L:smallT}.
\end{proof}

\section{Conclusion}\label{S:conclusion}

Our work above exactly determines the smallest number of stars in an uncompletable partial $k$-star design for all $k$ and $n$. One natural extension would be to investigate the same problem for \emph{$k$-star designs of index $\lambda$}, that is, $k$-star decompositions of $\lambda$-fold complete multigraphs. Since Lemma~\ref{L:realisation} can be generalised to multigraphs, many of the techniques we use here for finding completions could also be applied for designs of higher indices. Uncompletable partial designs in this setting could be obtained by taking $\lambda$-fold copies of the examples we use to prove Lemma~\ref{L:tightness}, but it is very possible that smaller uncompletable examples exist.

One could also consider analagous problems for other classes of graph design. A natural target would be $k$-cycle systems.

\bigskip\bigskip
\noindent\textbf{Acknowledgments.}
The second author was supported by Australian Research Council grants DP220102212 and
DP240101048. In the early stages of this work the first author was supported by a Postgraduate Publication Award from Monash University.



\begin{thebibliography}{99}

\bibitem{AndersonHilton1983}
    L.D. Anderson and A.J.W. Hilton,
    Thank Evans!,
    {\it Proc. London Math. Soc.} {\bf 47} (1983), 507--522.

\bibitem{Brock1988}
    B. W. Brock,
    Hermitian congruence and the existence and completion of generalized Hadamard matrices,
    {\it J. Combin. Theory Ser. A} {\bf 49} (1988), 233--261.

\bibitem{BryEtAl}
D. Bryant, N.J. Cavenagh, B. Maenhaut, K. Pula and I.M. Wanless, Non-extendible latin cuboids, {\it SIAM J. Discrete Math.} \textbf{26}, (2012), 239--249.

\bibitem{CameronHorsley2020}
    R. A. Cameron and D. Horsley,
    Decompositions of complete multigraphs into stars of varying sizes,
    {\it J. Combin. Theory Ser. B} {\bf 145} (2020), 32--64.


\bibitem{CaroSch1980}
Y. Caro and J. Sch\"{o}nheim,
Decomposition of trees into isomorphic subtrees,
{\it Ars. Combin.} {\bf 9} (1980), 119--130.


\bibitem{DeVHorEvans}
A. De Vas Gunasekara and D. Horsley, An Evans-style result for block designs, {\it SIAM
J. Discrete Math.} \textbf{36} (2022), 47--63.

\bibitem{DeVHorStar}
A. De Vas Gunasekara and D. Horsley, Smaller embeddings of partial $k$-star decompositions, {\it Electron. J. Combin.} \textbf{30} (2023), Paper No. 1.19, 20 pp.

\bibitem{Evans1960}
    T. Evans,
    Embedding incomplete Latin squares,
    {\it Amer. Math. Monthly} {\bf 67} (1960), 958--961.

\bibitem{Hoffman2004}
    D. G. Hoffman,
    The real truth about star designs,
    {\it Discrete Math.} {\bf 284} (2004), 177--180.

\bibitem{HoffmanRoberts2014}
    D. G. Hoffman and D. Roberts,
    Embedding partial $k$-star designs,
    {\it J. Combin. Des.} {\bf 22} (2014), 161--170.


\bibitem{NobNoc}
    M. Noble and S. Nochumson,
    Embedding partial 3-star designs,
    {\it Electron. J. Graph Theory Appl.} {\bf 12} (2024), 289--295.

\bibitem{NobleRichardson2019}
    M. Noble and S. N. Richardson,
    Balls, bins, and embeddings of partial $k$-star designs,
    {\it Discrete Math.} {\bf 342} (2019), 111600, 4 pp.

\bibitem{Smetaniuk1981}
    B. Smetaniuk,
    A new construction on Latin squares. I. A proof of the Evans conjecture,
    {\it Ars Combin.} {\bf 11} (1981), 155--172.

\bibitem{Tarsi1979}
    M. Tarsi,
    Decomposition of complete multigraphs into stars,
    {\it Discrete Math.} {\bf 26} (1979), 273--278.

\bibitem{Tarsi1981}
    M. Tarsi,
    On the decomposition of a graph into stars,
    {Discrete Math.} {\bf 36} (1981), 299--304.


\bibitem{Yamamoto1975}
    S. Yamamoto, H. Ikeda, S. Shige-eda, K. Ushio and N. Hamada,
    On claw-decomposition of complete graphs and complete bigraphs,
    {\it Hiroshima math. J.} (1975), 33--42.


\bibitem{YuLiu}
Q.R. Yu and G. Liu, Graph factors and matching extensions, Springer (2010).

\end{thebibliography}
\end{document}